\documentclass[10pt]{article}
\font\smallit=cmti10

\usepackage{amssymb,latexsym,amsmath,epsfig,amsthm} 
\usepackage{amsmath,amssymb,amsbsy,amsfonts,amsthm,latexsym,
         amsopn,amstext,amsxtra,euscript,amscd,amsthm, mathtools}
\usepackage[utf8]{inputenc}

\makeatletter

\renewcommand\section{\@startsection {section}{1}{\z@}
{-30pt \@plus -1ex \@minus -.2ex}
{2.3ex \@plus.2ex}
{\normalfont\normalsize\bfseries\boldmath}}

\renewcommand\subsection{\@startsection{subsection}{2}{\z@}
{-3.25ex\@plus -1ex \@minus -.2ex}
{1.5ex \@plus .2ex}
{\normalfont\normalsize\bfseries\boldmath}}

\renewcommand{\@seccntformat}[1]{\csname the#1\endcsname. }

\makeatother

\newtheorem{lem}{Lemma}

\newtheorem{prop}{Proposition}

\newtheorem{thm}{Theorem}

\newtheorem{rem}{Remark}
\newtheorem{remark}[rem]{Remark}

\def\N{{\mathbb N}}

\newcommand{\bfm}{\mathbf{m}}
\newcommand{\multim}{\binom{m}{m_1,\ldots,m_k}}

\begin{document}

\begin{center}
\uppercase{\bf Repetitions of multinomial coefficients and a generalization of Singmaster's conjecture}
\vskip 20pt
{\bf Jean-Marie de Koninck
}\\
{\smallit D\'epartement de math\'ematiques, Universit\'e Laval, Qu\'ebec,
Canada}\\
{\tt jmdk@mat.ulaval.ca}\\
\vskip 10pt
{\bf Nicolas Doyon
}\\
{\smallit D\'epartement de math\'ematiques, Universit\'e Laval, Qu\'ebec,
Canada}\\
{\tt Nicolas.Doyon@mat.ulaval.ca}\\
\vskip 10pt
{\bf William Verreault
}\\
{\smallit D\'epartement de math\'ematiques, Universit\'e Laval, Qu\'ebec,
Canada}\\
{\tt william.verreault.2@ulaval.ca}\\
\end{center}
\vskip 10pt

\centerline{\bf Abstract}
\noindent
    Given two integers $k\geq 2$ and $a>1$, let $N_k(a)$ stand for the number of multinomial coefficients, with $k$ terms, equal to $a$. 
    We
    study the behavior of $N_k(a)$ and show that its average and normal orders are equal to $k(k-1)$. We also prove that $N_k(a)=O\left((\log a/\log\log a)^{k-1}\right)$ and make several propositions about extreme results regarding large values of $N_k(a)$.

\pagestyle{myheadings}
\thispagestyle{empty}
\baselineskip=12.875pt
\vskip 30pt

\section{Introduction}

Let $N(a)$ denote the number of times the integer $a>1$ occurs as a binomial coefficient, that is, $\displaystyle N(a)=\# \left\{(n,r) \in \N^2 : \binom{n}{r}=a \right\}$. Singmaster \cite{kn:Singmaster-1971} conjectured that $N(a)=O(1)$ and proved that $N(a)=O(\log a)$. Abbott, Erd\H os, and Hanson \cite{kn:Abbott-Erdos-Hanson} later showed that $\displaystyle N(a)=O\left(\frac{\log a}{\log\log a}\right)$ and that its average and normal orders are 2. Finally, Kane \cite{kn:Kane-2004}, \cite{kn:Kane-2007} improved the bounds to $\displaystyle O\left(\frac{\log a\log\log\log a}{(\log\log a)^2}\right)$ and $\displaystyle O\left(\frac{\log a\log\log\log a}{(\log\log a)^3}\right)$, respectively.

Considering the properties of Pascal's triangle, one can reformulate the problem as follows. Is there an upper bound for the number of times a given integer appears in Pascal's triangle?

Despite the fact that infinitely many positive integers appear at least six times in Pascal's triangle \cite{kn:Singmaster-1975}, only nine such integers, namely $$1, 120, 210, 1540, 3003, 7140, 11\,628, 24\,310 \text{ and } 61218182743304701891431482520,$$ have been found to satisfy this property up to $10^{60}$ \cite{kn:OEIS}. It is worth noting that $N(3003)=8$, and that this is the highest known value of $N(a)$. Singmaster \cite{kn:Singmaster-1975} actually conjectured that its maximal value might be $8, 10$ or $12$. 

In what follows, we study repetitions of multinomial coefficients, that is, for a fixed integer $k\geq 2$, those numbers
$\displaystyle\multim=\frac{m!}{m_1!\cdots m_k!}$, where $\displaystyle m=\sum_{i=1}^k m_i$.

Given an integer $k\geq 2$, let $\bfm=(m_1,\ldots,m_k) \in \N^k$, and let
\begin{equation*}
    N_k(a)=\# \left\{\bfm \in \N^k : \multim = a \right\}
\end{equation*}
denote the number of times the integer $a>1$ occurs as a multinomial coefficient. As such, $N_k(a)$ is a generalized version of $N(a)$ where we consider multinomial coefficients of $k$ terms equal to $a$, and in particular, $N(a)=N_2(a)$. 

If one considers Pascal's $k$-simplices as a generalization of Pascal's triangle in dimension $k$, where the entries of the simplex are given by multinomial coefficients of $k$ terms, then the problem is equivalent to asking whether there is an upper bound on the number of times an integer appears in Pascal's $k$-simplex for a given $k\geq 2$.

A simple program tells us that in the first $1000$ lines of Pascal's triangle, $494$ integers appear once, $248\,861$ twice, $5$ three times, $63$ four times and $3$ six times. In comparison, Pascal's $3$-simplex, often called Pascal's pyramid, has $445\,666$ distinct integers in its first $250$ layers, of which $429\,135$ appear exactly six times. These results were to be expected, as we will see.

Obviously, it may be the case that for any fixed value of $k\geq 2$, $N_k(a)$ is bounded. We will also pursue this investigation for large values of $k$.

\section{Main results}

We start by showing that the average and normal orders of $N_k(a)$ are both $k(k-1)$. Intuitively, the average order of $N_k(a)$ is another arithmetic function which takes the same values on average, while its normal order is an arithmetic function which is close to $N_k(a)$ for almost all values of $a$. Consequently, our first theorem implies that for any given integer $a > 1$, there are almost always $k(k-1)$ multinomial coefficients with $k$ terms equal to $a$. An immediate corollary is that almost always one can find any given integer a total of $k(k-1)$ times in Pascal's $k$-simplex.

\begin{remark} \label{rem:1}
    A multinomial coefficient of the form $\displaystyle \binom{a}{a-1,1,0,\ldots,0}$ is always equal to $a$. As such, one gets $\displaystyle N_k(a) \geq k(k-1)$ for all $k \geq 2$. Now notice that $a=1$ and $a=2$ are the only exceptions to this, and while $1$ appears infinitely many times, 2 occurs $\displaystyle \binom{k}{2}$ times as a multinomial coefficient, for it must be of the form $\displaystyle \binom{2}{1,1,0,\ldots,0}$.
\end{remark}

\begin{thm} \label{thm:}
    Let $k\geq 2$ be a fixed integer and $N_k(a)$ be defined as above. Then, the average and normal orders of $N_k(a)$ are both $k(k-1)$.
\end{thm}

\begin{proof}
    Let $M$ be a positive integer. We define 
    \begin{align*}
        \mathcal{S}(k,M) &= \left\{\bfm : 2 < \multim \leq M \right\}, \\
        \mathcal{S}_1(k,M) &= \left\{\bfm\in \mathcal{S}(k,M) : m_i=m-1 \text{ for some } i \in \{1,\ldots,k\}\right\}, \\
        \mathcal{S}_2(k,M) &= \mathcal{S}(k,M)\setminus \mathcal{S}_1(k,M).
    \end{align*}
    
    Following Remark \ref{rem:1}, we find that
    \begin{align}
        \sum_{1<a\leq M}N_k(a) 
        &= \binom{k}{2} \: + \sum_{\bfm \in \mathcal{S}(k,M)}1 \label{eq:50} \\
        &= \binom{k}{2} \: + \: \sum_{\bfm \in \mathcal{S}_1(k,M)}1 \quad
        + \sum_{\bfm \in \mathcal{S}_2(k,M)}1. \nonumber
    \end{align}
    Now, since
    \begin{align} 
        \sum_{\bfm \in \mathcal{S}_1(k,M)}1 &= k(k-1)\sum_{2<\binom{m}{1,0,\ldots,0,m-1}\leq M}1 \label{eq:51} \\
        &=k(k-1)\sum_{2 < m\leq M}1 \nonumber\\
        &=k(k-1)(M-2), \nonumber
    \end{align}
    it only remains to consider $\displaystyle\sum_{\bfm \in \mathcal{S}_2(k,M)}1$.
    
    If $\bfm \in \mathcal{S}_2(k,M)$, then 
    $\displaystyle 2 < \multim \leq M$ and $m_i < m-1$ for $i=1,2,\ldots,k$. Assume, without loss of generality, that $m_i \leq m_{i+1}$ for $i=1,2,\ldots,k-1$. 
    Then we have the inequalities
    \begin{equation} \label{eq:4}
      M\geq \multim \geq \frac{m!}{m_j!(m-m_j)!} \geq \binom{2m_j}{m_j} \geq 2^{m_j} \quad (j=1,2,\ldots,k-1)
    \end{equation}
    and
    \begin{equation*}
      M\geq \multim \geq \frac{m!}{m_k!(m-m_k)!} \geq \binom{m}{2}=\frac{m(m-1)}{2},
    \end{equation*}
    which implies that $m_j=O(\log M)$ and $m=O(M^{1/2})$, respectively.
    
    Since the value of $m_k$ is entirely determined by that of $m,m_1,m_2,\ldots,m_{k-1}$, we get
    \begin{equation} \label{eq:52}
        \sum_{\bfm \in \mathcal{S}_2(k,M)}1=\sum_{\substack{2<\multim\leq M \\ 0\leq m_i < m-1}}1
        = \sum_{\substack{m_i \\ 1\leq i \leq k-1}}\sum_{m}1 = O\left(M^{1/2}\left(\log M\right)^{k-1}\right).
    \end{equation}

    Using \eqref{eq:51} and \eqref{eq:52} in \eqref{eq:50}, we obtain
    \begin{equation} \label{eq:1}
        \sum_{1<a\leq M}N_k(a)=\binom{k}{2}+k(k-1)(M-2)+O\left(M^{1/2}\left(\log M\right)^{k-1}\right),
    \end{equation}
    and therefore,
    \begin{align*}
        \lim_{M\to\infty} \frac{1}{M} &\sum_{1<a\leq M}N_k(a) \\
        &= \lim_{M\to\infty} \frac{1}{M} \left(\binom{k}{2}+k(k-1)(M-2) + O(M^{1/2}(\log M)^{k-1})\right) \\
        &=k(k-1),
    \end{align*}
        which provides the average order of $N_k(a)$.
         
        For the normal order, let
    \begin{align*}
        f_k(M) &= \# \left\{1<a\leq M : N_k(a) < k(k-1)\right\}, \\
        g_k(M) &= \# \left\{1<a\leq M : N_k(a) = k(k-1)\right\}, \\
        h_k(M) &= \# \left\{1<a\leq M : N_k(a) > k(k-1)\right\}.
    \end{align*}
    We wish to prove that $\displaystyle N_k(a)-k(k-1) \leq \varepsilon k(k-1)$ for all $\varepsilon >0$ and for almost all $a$. It should be clear that this inequality holds only if  $N_k(a)=k(k-1)$, for $N_k(a)$ only takes integral values. As such, it suffices to show that $f_k(M)+h_k(M)= o(M)$ as $M\to\infty$.
    
    Now Remark \ref{rem:1} implies that $f_k(M)=1$. Also, $g_k(M) + h_k(M) = M-2$ since they only exclude $a=1$ and $a=2$. Thus,
    \begin{align}
        \sum_{1<a\leq M}N_k(a) &\geq f_k(M) + k(k-1)g_k(M) + (k(k-1)+1)h_k(M) \nonumber \\
        &= 1 + k(k-1)(g_k(M)+h_k(M)) + h_k(M) \nonumber \\
        &= 1 + k(k-1)(M-2) + h_k(M). \label{eq:2}
    \end{align}
    Comparing \eqref{eq:2} with \eqref{eq:1}, we get that $h_k(M)=O\left(M^{1/2}(\log M)^{k-1}\right)$, which means that $f_k(M)+h_k(M)=o(M)$ as $M\to\infty$ and in turn implies that the normal order of $N_k(a)$ is indeed $k(k-1)$. \qedhere
\end{proof}

We also state an additional result, even though it is weaker than our claim  
$$N_k(a)=O\left(\left(\frac{\log a}{\log\log a}\right)^{k-1}\right),$$ 
be it simply for the fact that our proof is both more general and simpler than that of Singmaster in \cite{kn:Singmaster-1971}.

\begin{prop} \label{thm:2}
    Let $k\geq 2$ be a fixed integer. Then $\displaystyle N_k(a)=O\left(\log^{k-1} a\right).$
\end{prop}

\begin{proof}
    Since
    \begin{equation*} 
        \frac{\binom{m+1}{m_1,\ldots,m_j+1,\ldots,m_k}}{\multim} 
        = \frac{\frac{m+1}{m_j+1}\multim}{\multim}
        = \frac{m+1}{m_j+1}\geq 1,
    \end{equation*}
    the multinomial coefficients are strictly increasing in $m_j$ for all $1\leq j \leq k$, which implies 
    that if we fix $k-1$ parameters in the multinomial coefficient, say all but $m_i$, then
    \begin{equation} \label{eq:20}
        \# \left\{m_i : \multim = a\right\} \leq 1. 
    \end{equation}
    
    The proof is based on the observation that, for all $1\leq i\leq k$,
    \begin{equation*}
        N_k(a)=\sum_{\substack{m_j \\ j \: \in \: \{1,\ldots,k\} \setminus \{i\}}}\# \left\{m_i : \multim = a \right\}, 
    \end{equation*}
    so that in particular,
    \begin{equation*}
        N_k(a)=\sum_{\substack{m_j \\ 1\leq j \leq k-1  }}\# \left\{m_k : \multim = a \right\} 
        \leq \sum_{\substack{m_j \\ 1\leq j \leq k-1  }}1
        =O(\log^{k-1} a),
    \end{equation*} 
   where we used \eqref{eq:20} and \eqref{eq:4}.
\end{proof}

\begin{remark}
  One may be tempted to write $N_k(a)$ as a sum over less than $k-1$ fixed parameters, but then it becomes quite a tedious task to determine the number of multinomial coefficients equal to $a$. This is why we should be satisfied with this bound for now.
\end{remark}

Next, we generalize a result of Abbott, Erd\H os, and Hanson \cite{kn:Abbott-Erdos-Hanson}, obtaining a better upper bound for $N_k(a)$. We will need two lemmas. 

\begin{lem}[Baker, Harman, and Pintz \cite{kn:Baker-Harman-Pintz}]  \label{lem:2}
    There exists a real number $x_0$ such that the interval $\displaystyle \left[x,x+x^{0.525}\right]$ contains at least one prime number for all $x>x_0$.
\end{lem}

\begin{lem} \label{lem:1}
    For each $j=1,2,\ldots,k$, we have $\displaystyle \multim\geq \left(\frac{m}{m_j}\right)^{m_j}$.
\end{lem}

\begin{proof}
    We proceed by induction on $m_j$. 
    At first, if $m_j=1$, then, 
    \begin{align*}
        \multim &= \binom{m_1+\cdots+m_{j-1}+1+m_{j+1}+\cdots+m_k}{m_1,\ldots,m_{j-1},1,m_{j+1}\ldots,m_k} \\
        &= (m_1+\cdots+m_{j-1}+1+m_{j+1}+\cdots+m_k)\binom{m_1+\cdots+m_{j-1}+m_{j+1}+\cdots+m_k}{m_1,\ldots,m_{j-1},m_{j+1},\ldots,m_k} \\
        &\geq (m_1+\cdots+m_{j-1}+1+m_{j+1}+\cdots+m_k).
    \end{align*}
    
    Now,
    \begin{align*}
        \binom{m+1}{m_1,\ldots,m_j+1,\ldots,m_k} &= \frac{m+1}{m_j+1}\multim \\
        &\geq \frac{m+1}{m_j+1} \left(\frac{m}{m_j}\right)^{m_j} \\
        &\geq \frac{m+1}{m_j+1} \left(\frac{m+1}{m_j+1}\right)^{m_j} \\
        & = \left(\frac{m+1}{m_j+1}\right)^{m_j+1},
    \end{align*}
    where the second to last inequality holds since $m\geq m_j$.
\end{proof}

\begin{thm} \label{thm:3}
    Let $k\geq 2$ be a fixed integer. Then, $\displaystyle N_k(a)=O\left(\left(\frac{\log a}{\log\log a}\right)^{k-1}\right).$
\end{thm}

\begin{proof} We let $m_i \leq m_{i+1}$ for $i=1,2,\ldots,k-1$ without loss of generality.

    Let
        \begin{align*}
        A &= \left\{\bfm \in N_k(a) : m > \log^{\frac{22}{21}} a\right\}, \\
        B &= \left\{\bfm \in N_k(a) : m \leq \log^{\frac{22}{21}} a\right\},
    \end{align*}
    so that $N_k(a)=\#A+\#B$.
    Hence, it is sufficient to show that $\#A$ and $\#B$ are both $$\displaystyle O\left(\left(\frac{\log a}{\log\log a}\right)^{k-1}\right).$$

    First, assume $\bfm$ is in $A$, so that 
    \begin{equation} \label{eq:9}
        m>\log^\frac{22}{21} a.
    \end{equation}
    
    Notice that in \eqref{eq:4}, we have not used the fact that $m_i < m-1$ for all $1\leq i\leq k$. Thus, we can write
    \begin{equation} \label{eq:10}
    m_i\leq \frac{\log a}{\log 2} \qquad (i=1,2,\ldots,k-1).
    \end{equation}

    Using Lemma \ref{lem:1}, \eqref{eq:9} and \eqref{eq:10}, we get
    \begin{equation*}
        m_i \leq \frac{\log a}{\log \frac{m}{m_i}}
        < \frac{\log a}{\log \frac{\log^\frac{22}{21} a}{\log a/\log 2}} 
        =O\left(\frac{\log a}{\log\log a}\right) \qquad (i=1,2,\ldots,k-1).
    \end{equation*}
    
    By monotonicity, for fixed $m_1,\ldots,m_{k-1}$, there is only one $m_k$ that yields a multinomial coefficient equal to $a$. Hence,
    \begin{equation} \label{eq:8}
        \# A=O\left(\left(\frac{\log a}{\log\log a}\right)^{k-1}\right).
    \end{equation}

                Secondly, assume $\bfm$ is in $B$. Let $M=M_1+M_2+\cdots +M_k$ be the greatest element for which the associated $\mathbf{M}=(M_1,\ldots,M_k)$ lies in $B$. We also define $S=M-M_1$.
                Then,
                \begin{equation*}
                    a=\frac{M!}{M_1!\cdots M_k!}<\frac{M!}{(M-S)!}=M(M-1)\cdots (M-S+1)<M^S.
                \end{equation*}
                
                From $\displaystyle M \leq \log^\frac{22}{21} a$ and $a < M^{S}$, we get
                \begin{equation*}
                M < (S\log M)^\frac{22}{21}.
                \end{equation*}
                
                Note that 
                \begin{equation*}
                    \log^\frac{22}{21} M \leq S^\frac{18}{21} \text{ for all }  M\geq 2.
                \end{equation*}
                
                Since $k\geq 2$ and $M_1$ is maximal when all parameters are equal, it is sufficient to show that the result holds for $M_1=M/2$. It is easy to verify that $\log^\frac{22}{21}M\leq (M/2)^\frac{18}{21}$ for all $M\geq 2$.

                Therefore, $\displaystyle M \leq S^{22/21}S^{18/21} < S^{40/21}+ S$ or, equivalently, $\displaystyle (M-S)^{21/40} < S$.
                 Adding $M-S$ on both sides, we finally have 
               $\displaystyle
                    (M-S)^{21/40}+(M-S) < M,
                $
                that is,
                \begin{equation} \label{eq:5}
                    M_1+M_1^{21/40} < M.
                \end{equation}
                
                Using Lemma \ref{lem:2}, there exists a prime in $\displaystyle \left[M_1, M_1+M_1^\frac{21}{40}\right]$. Combining the latter with \eqref{eq:5}, there is a largest prime P satisfying
                \begin{equation*}
                    M_1 \leq P < M,
                \end{equation*}
                which tells us that $P$ divides our fixed $a$, so the same can be said of any $m$ for which the associated $\bfm$ is in $B$. Thus, $P\leq m\leq M$ whenever $m\leq \log^\frac{22}{21} a$.
                As such, there are at most $M-P$ elements in $B$. 
                
                Next, let $x=P$ in Lemma \ref{lem:2} and let $Q$ be the largest prime in the interval $\displaystyle \left[P, P+P^\frac{21}{40}\right]$. Then, since $P$ is the largest prime less than $M$, we must have
                \begin{equation*}
                    P<M\leq Q \leq P+P^{21/40},
                \end{equation*}
                and so it follows that
                \begin{equation*}
                    \# B \leq M-P \leq P^\frac{21}{40} < M^\frac{21}{40} 
                    \leq (\log^\frac{22}{21}a)^\frac{21}{40} = \log^\frac{11}{20} a=O\left(\left(\frac{\log a}{\log\log a}\right)^{k-1}\right),
                \end{equation*}
                which, combined with \eqref{eq:8}, yields the desired result.
\end{proof}

    

We end this paper by stating some partial results regarding large values of $N_k(a)$. We will see that we can do much better than the naive bound $N_k(a) \geq k(k-1)$ for all $k\geq 2$.

\begin{prop} \label{prop:2}
For any integer $k>2$,
$$
\#\{a \leq x:N_k(a)\ge k! \}\gg x^{\frac{2}{(k-1)(k-2)}}.
$$
\end{prop}
\begin{proof}
We will force $a$ to take the form of a multinomial coefficient by counting the integers $a$ that can be written as
\begin{equation} \label{eq:25}
  a=\frac{(m+(k-1)(k-2)/2)!}{0!1!\cdots(k-2)! m!}\leq x,
\end{equation}
which obviously has at least $k!$ representations as a multinomial coefficient once $x$ is big enough, by permuting the terms in the denominator.
We have
$$
\frac{(m+(k-1)(k-2)/2)!}{0!1!\cdots(k-2)! m!}=\prod_{j=1}^{k-2}\frac{1}{j!} \prod_{j=1}^{(k-1)(k-2)/2} (m+j)\sim C_k m^{\frac{(k-1)(k-2)}{2}}
$$
as $m\to\infty$,
where 
$$
C_k=\prod_{j=1}^{k-2}\frac{1}{j!}.
$$
It follows that for $x$ large enough, \eqref{eq:25} holds as long as
$$
(1+\varepsilon)C_km^{(k-1)(k-2)/2}\leq x$$ for any given $\varepsilon>0$, that is
\begin{equation*}
  m\leq \frac{1}{((1+\varepsilon)C_k)^{\frac{2}{(k-1)(k-2)}}}x^{\frac{2}{(k-1)(k-2)}},
\end{equation*}
which proves our claim.
\end{proof}

\begin{prop} \label{prop:3}
For each $k\ge 4$,
$$
\#\{a:N_k(a)\ge 2\cdot k!+k(k-1)\} = \infty.
$$
\end{prop}
Note that this proposition also works for $k=2$ since infinitely many positive integers appear at least six times in Pascal's triangle. 

\begin{proof}
Our proof follows from the simple identity 
\begin{equation}\label{eq1}
12!2!1!=11!4!0!,
\end{equation}
from which we have, for any $m\ge 15$,
$$
a(m):=\frac{m!}{12!2!1!(m-15)!}=\frac{m!}{11!4!0!(m-15)!}.
$$
If $m\ge 28$, all the integers appearing in the denominator of the above equation are distinct, thus providing $2\cdot 4!=48$ distinct vectors. To obtain the remaining $4\cdot 3=12$ vectors, we just have to consider the expression
$$
\frac{a(m)!}{(a(m)-1)!1!0!0!},
$$
thus completing the proof of the proposition when $k=4$.
This idea can easily be generalized for $k>4$ by observing that
\begin{equation*}
\frac{(m+m_5+m_6+\cdots + m_k)!}{12!2!1!(m-15)!m_5!m_6!\cdots m_k!}=\frac{(m+m_5+m_6+\cdots m_k)!}{11!4!0!(m-15)!m_5!m_6!\cdots m_k!}. \qedhere
\end{equation*}
\end{proof}

\begin{prop} \label{prop:4}
For each $k\ge 7$,
$$
\#\{a:N_k(a)\ge 7\cdot k!/2 +k(k-1)\}=\infty.
$$
\end{prop}

\begin{proof}
First observe the identity
\begin{equation}\label{eq2}
24!5!3!1!=23!6!4!0!.
\end{equation}
From equations \eqref{eq1} and \eqref{eq2}, we have
$$
24!5!3!1!12!2!=23!6!4!0!12!2!=24!5!3!11!4!0!=23!6!4!0!11!4!0!,
$$
from which we can deduce that for any $m\ge 47$,
\begin{align*}
&\frac{m!}{24!12!5!3!2!1!(m-47)!}=\frac{m!}{23!12!6!4!2!0!(m-47)!}
=\frac{m!}{24!11!5!4!3!0!(m-47)!} \\
&=\frac{m!}{23!6!4!0!11!4!0!(m-47)!}.
\end{align*}
By choosing $m\ge 72$, we ensure that all the integers appearing in the denominator are distinct, except $0$ in the last denominator, which provides $7\cdot 7!/2=17\,640$ solutions. To obtain the $7\cdot 6=42$ remaining solutions, we set
$$
a(m)=\frac{m!}{24!12!5!3!2!1!(m-47)!}
$$
and consider the quantity
$$
\frac{a(m)!}{(a(m)-1)!1!0!0!0!0!0!}.
$$
This argument can be generalized for $k>7$, by observing that
\begin{eqnarray*}
& &\frac{(m+m_8+m_9+\cdots+m_k)!}{24!12!5!3!2!1!(m-47)!m_8!m_9!\cdots m_k!}\\
&=&\frac{(m+m_8+m_9+\cdots+m_k)!}{23!12!6!4!2!0!(m-47)!m_8!m_9!\cdots m_k!}\\
&=&\frac{(m+m_8+m+m_9+\cdots+ m_k)!}{24!11!5!4!3!0!(m-47)!m_8!m_9!\cdots m_k!}\\
&=&\frac{(m+m_8+m+m_9+\cdots+ m_k)!}{23!6!4!0!11!4!0!(m-47)!m_8!m_9!\cdots m_k!}.
\end{eqnarray*}
From this we have that for $k\ge 7$,
\begin{equation*}
\#\{a:N_k(a)\ge 7\cdot k!/2 +k(k-1)\}=\infty. \qedhere
\end{equation*}
\end{proof}

We also looked numerically for large values of $N_3(a)$ and $N_4(a)$. With $a=2671465728531600$, we found that $N_3(a)\ge 30$ and $N_4(a)\ge 180$. Recall that we expect $N_3(a)=6$ and $N_4(a)=12$. The bound $N_3(a)\ge 30$ comes from the fact that
$$
a=\frac{37!}{7!11!9!}=\frac{38!}{19!11!8!}=\frac{39!}{19!14!6!}=\frac{40!}{19!16!5!}=\frac{a!}{(a-1)!1!0!}.
$$
To prove the bound $N_4(a)\ge 180$, it suffices to observe that
\begin{eqnarray*}
a&=&\frac{37!}{17!11!9!0!}=\frac{38!}{19!11!8!0!}=\frac{39!}{19!14!6!0!}\\
&=&\frac{40!}{19!16!5!0!}=\frac{40!}{20!16!3!1!}=\frac{39!}{21!13!4!1!}=\frac{38!}{22!10!4!2!}=\frac{a!}{(a-1)!1!0!0!}.
\end{eqnarray*}

\section{Concluding remarks}
    
    At this point, it seems natural to ask whether $N_k(a)=O(1)$ for all $k\geq 2$. However, the problem seems just as hard as in the basic case of binomial coefficients. Obviously, as multinomial coefficients can be rewritten as a product of binomial coefficients, it would be particularly interesting to show that proving Singmaster's conjecture is equivalent to proving its generalization in any dimension $k\geq 2$. Perhaps considering symmetries and patterns in Pascal's $k$-simplex would help. For example, setting a parameter to $0$ in $\bfm$ leads us to Pascal's $(k-1)$-simplex, and induction seems to be the way forward.
    
    Our final remark is that there is still a lot we don't know about repetitions of multinomial coefficients. For instance, whether $N_2(a)=5$ or $N_2(a)=7$ is even possible.
    Furthermore, Singmaster \cite{kn:Singmaster-1975} showed that
    \begin{equation*}
        \binom{F_{2i+2}F_{2i+3}}{F_{2i}F_{2i+3}}=\binom{F_{2i+2}F_{2i+3}-1}{F_{2i}F_{2i+3}+1},
    \end{equation*}
    where $F_n$ is the $n$th Fibonacci number, by solving a type of Pell equation. The same result has been proved in \cite{kn:Lind}. This implies that infinitely many integers are going to appear at least six times in Pascal's triangle, as stated in the Introduction.
    Perhaps similar considerations, such as the ones we made in Propositions \ref{prop:2}, \ref{prop:3} and \ref{prop:4}, can be made in higher dimensions to find equal multinomial coefficients, although equations in several variables would have to be solved in a different manner.

\end{document}